\documentclass[final,12pt]{article}
\usepackage{amsfonts}
\usepackage{amssymb}
\usepackage{amsmath}
\usepackage{amsthm}
\usepackage{natbib}
\usepackage{setspace}
\usepackage{fancyhdr}
\usepackage{titlesec}
\usepackage{color}
\usepackage[paperwidth=8.5in,left=1in,right=1in,top=1in,bottom=1in,paperheight=11in]{geometry}
\usepackage{pstricks}
\usepackage{pst-plot}
\usepackage{pst-node}
\DeclareMathOperator*{\E}{\mathbb{E}}

\usepackage[colorlinks=true, citecolor=black, linkcolor=darkblue, pagebackref=false,breaklinks=true]{hyperref}
\usepackage{titlesec}
\usepackage{color}
\usepackage{varioref}

\definecolor{darkblue}{rgb}{0,0,0.55}
\definecolor{darkred}{rgb}{0.5,0,0}
\newcommand\Bheadfont{\fontsize{14pt}{\baselineskip}\selectfont}

\definecolor{darkblue}{rgb}{0,0,0.55}
\definecolor{darkred}{rgb}{0.5,0,0}

\titleformat{\section}[hang] {\normalfont\sc\color{darkblue}\Bheadfont} {\thesection\hskip0.618em}{0em}{}
\titlespacing*{\section}
{0pt}{15pt plus 2pt minus 2pt}{9pt plus 2pt minus 2pt}
\titleformat{\subsection}[runin]
{\normalfont\sc\color{darkblue}} {\thesubsection\hskip0.618em}{0em}{}
\titlespacing*{\subsection}
{0pt}{13pt plus 2pt minus 2pt}{13pt plus 2pt minus 2pt}
\titleformat{\subsubsection}[runin]
{\normalfont\sc\color{darkblue}} {\thesubsubsection\hskip0.618em}{0em}{}
\titlespacing*{\subsubsection}
{0pt}{13pt plus 2pt minus 2pt}{13pt plus 2pt minus 2pt}

\theoremstyle{plain} 
\theoremstyle{plain} \newtheorem{proposition}{Proposition}
\theoremstyle{remark} 
\theoremstyle{plain} \newtheorem{lemma}[proposition]{Lemma}
\theoremstyle{plain} 
\theoremstyle{plain} 
\theoremstyle{plain} 
\theoremstyle{plain} 

\newcommand{\theoremname}{Theorem}
\newcounter{thm}
\newenvironment{thm}{%
  \par\medskip\refstepcounter{thm}%
  \noindent\textcolor{darkred}{\textbf{Theorem \thethm}}:}{\par\medskip}
\labelformat{thm}{\theoremname~#1}

\begin{document}
\title{\vspace*{1in}{\sc{\large{Approximate Aggregation in the Neoclassical Growth Model with Idiosyncratic Shocks}\thanks{Katz acknowledges support from NSF grant DMS 1266104.}}}}
\author{Karsten O. Chipeniuk\thanks{Department of Mathematics, California Institute of Technology, \texttt{kchipeni@caltech.edu}} \; \; \; \; 
Nets Hawk Katz\thanks{Department of Mathematics, California Institute of Technology, \texttt{nets@caltech.edu}} \;\;\;\; 
Todd B. Walker\thanks{Department of Economics, Indiana University, \texttt{walkertb@indiana.edu}}}
\date{March 2014\\Preliminary and Incomplete} 
\maketitle

\vspace{1in}
\begin{abstract}
	We provide an explicit characterization of aggregation in the neoclassical growth model with aggregate shocks and uninsurable employment risk. We show that there are two restrictions on the employment shock process that must be satisfied in order for approximate aggregation to hold. First, the probability of unemployment must be positive for each agent every period. This ensures a strong precautionary savings motive. Second, like agents must have similar future prospects. That is, agents with similar employment shocks and wealth distributions cannot have drastically different employment paths. The model solution requires the distribution of wealth as a relevant state variable, and hence the curse of dimensionality must be confronted. We sidestep this thorny issue by introducing a Walrasian auctioneer that communicates the optimal amount of aggregate capital in each period for every shock to the agents. 

\bigskip

\end{abstract}

\bgroup

\let\footnoterule\relax 

\thispagestyle{empty} 
\egroup
\vspace{0.5in} 
\phantom{s} 

\pagebreak 
\onehalfspacing
\setcounter{page}{1}
\section{Introduction}

The one-sector neoclassical growth model with uninsurable idiosyncratic shocks has become a workhorse in macroeconomics. Early versions of the model were used to examine the role of incomplete insurance on the permanent income hypothesis, monetary and fiscal policy, the cost of business cycles, asset pricing, etc. [\cite{Bewley:77}, \cite{Bewley:86}, \cite{Imrohoroglu:89}, \cite{Huggett:93}, \cite{Aiyagari1994}, \cite{MarcetSingleton:99}]. These papers spawned a vast literature which has demonstrated the importance of the incomplete markets assumption, and led to the popularity of this model as an important research tool.  

Despite the popularity, few analytical results have been established with respect to the model's properties. Perhaps the best example of this pertains to aggregation. The key message of \cite{KrusellSmith1998} (KS, hereafter) is that the neoclassical growth model with idiosyncratic risk and aggregate shocks features ``approximate aggregation." Using a numerical approach that is now well known, KS show that most agents can self-insure through the accumulation of capital. These agents have nearly affine policy functions in the state variables, which permits the aggregation of \cite{Gorman:53,Gorman:61}.  In equilibrium, there exists a small fraction of agents who are close to their borrowing constraints, but their overall contribution to the aggregate capital stock is so small that it is nearly negligible; hence \emph{approximate} aggregation attains. While this aggregation result is a robust numerical finding, a formal treatment is missing. This paper serves to fill this void.

Our model is that of KS with a finite number of agents and time periods. The primary challenge in solving this model is that the future price of capital depends upon the asset holdings and employment status of each agent. Therefore, the distribution of wealth is a relevant state variable. As the number of agents increases, the curse of dimensionality takes hold and the problem becomes untenable. 

In order to circumvent this issue, we introduce a Walrasian auctioneer who sets the level of aggregate capital in advance for all time periods and all outcomes of the shocks. This sequence is then communicated to the agents. Armed with this knowledge, agents are free to make investment-consumption allocations without having to know the distribution of wealth. We show that the typical agent's problem has a unique solution, but we do not take up the challenge of solving the Walrasian auctioneer's problem in this paper. We therefore can make no claims about the existence or uniqueness of our equilibrium. However, this setup does allow us to provide a complete characterization of aggregation. 

Our main theorem delivers conditions under which the model of KS approximately aggregates. We show that there are two important restrictions on the employment shock necessary to achieve aggregation. First, in every period, the probability of unemployment must be positive. Capital must serve the dual role of being a store of wealth, facilitating intertemporal substitution of consumption, and providing insurance against employment shocks. Precautionary savings must be sufficiently high in order for the latter condition to be met effectively. The risk of unemployment each period ensures that this is the case. Second, agents of a similar background must have similar future prospects.  Agents with nearly identical wealth profiles but very different employment prospects will not have similar consumption profiles. We place a bound on the extent to which agents' employment paths are dissimilar. These restrictions are the only ones imposed on the employment shocks. Thus, our results nest those of KS.

\section{The Environment}

\subsection{Time} 
Time is discrete and finite, consisting of $T$ periods and indexed by $t$. We will use the convention that a new date commences with the arrival of new information. Any variable known or chosen at date $t$ will be indexed by $t$. 

\subsection{Firms} Factor and product markets are perfectly competitive. The aggregate production technology is Cobb-Douglas, $Y_t=F(K_{t-1},L_t)=z_t K_{t-1}^\alpha L_t^{1-\alpha}$,
with $\alpha \in [0,1]$ and aggregate productivity shock, $z_t$.  Profit maximization delivers the rental rate of capital and the wage rate as 
\begin{gather}
R_t = \alpha z_t  \bigg(\frac{K_{t-1}}{L_t}\bigg)^{\alpha-1} \label{eq:RentalRate}\\
w_t =  (1-\alpha) z_t\bigg(\frac{K_{t-1}}{L_t}\bigg)^\alpha  \label{eq:WageRate}
\end{gather}

\subsection{Households}
There are $N$ households indexed by $j$ that live for $T$ periods. Each household values consumption according to 
\begin{gather}
U(c_1,...,c_T) = \mathbb{E}_0 \sum_{t=1}^T \beta^{t-1} u(c_t) \label{eqn:Utility}
\end{gather}
where $\beta \in (0,1)$ is the discount factor. We assume that the instantaneous utility function takes the form of constant relative risk aversion (CRRA) $u(c_t) = (c_t^{1-\sigma}-1)/(1-\sigma)$, where it is understood 
that $u(c_t)=\log c_t$ for $\sigma=1$. 

Agents are endowed with one unit of time each period and do not value leisure. The household's units of labor supplied to the market in period $t$, $\xi_t$, will be stochastic and lie on the unit interval to allow for unemployment, full employment and under-employment. Let the random variable $e_{j,t} \in [0,1]$ denote the share of the wage bill that belongs to household $j$ at time $t$.  The possible values of $e_{j,t}$ and the probabilities that they occur constitute the employment shock. These employment shocks will have a few restrictions but will otherwise be general enough to nest the standard assumptions found in the literature [e.g., KS].  For now, it is enough to know that $e_{j,t}$ will be allowed to depend on both previous aggregate states $z_{t'}$ and employment states $e_{j,t'}$, $t'\leq t$.

Households have access to an asset, $a_t$, with rate of return on time-$t$ holdings of $r_t$. The flow budget constraint for household $j$ can be written as $c_{j,t} = w_t L_t e_{j,t} + (1+r_t) a_{j,t-1} - a_{j,t}$, where $a_{j,t}$ is the $t$-period asset holding for household $j$ that yields return $r_{t+1}$.  We assume that the agents cannot borrow and therefore restrict asset holdings to be strictly positive. That is, we impose the natural borrowing limit.    Market clearing will impose $a_t = k_t$ and we can rewrite the budget constraint as 
\begin{gather}
c_{j,t} =w_t L_t e_{j,t} + (1+R_t-\delta) k_{j,t-1}- k_{j,t} \label{eqn:BC1}
\end{gather}
where  $\delta \in (0,1)$ is the rate of depreciation of capital, and the household's return on asset holdings is $r_t = R_t-\delta$.

\subsection{The Market Arrangement} 

In order to make their decisions rationally, the individual agents must account for the
choices of their peers. This is important because the aggregate capital stock determines
the price paid for capital, and hence the return on investment. The aggregate capital stock will be a function of the wealth of each agent. A fully rational approach would have each agent tracking the distribution of wealth over time, as a relevant aggregate state variable.  If $N$ is ``big", the curse of dimensionality takes hold and this becomes untenable.

Instead, we introduce a Walrasian auctioneer (henceforth, the auctioneer) 
to mediate between the agents. She is the only one who needs or uses knowledge
of the distribution of wealth. Her job is to set the aggregate levels of the capital stock
$K_{t}$, $1\leq t\leq T$ in advance for all time periods and all outcomes of the 
shocks. The auctioneer announces her choices to each of the agents. They accept her forecasts as gospel. This is the extent of their bounded rationality.

There is only one restriction on how the auctioneer may set her forecasts. They
must invariably come true. That is, for any $t$ and any outcome of preceding shocks,
it must be the case that $K_{t}= \sum_{j=1}^N k_{j,t},$.
Here the individual investments on the right are the solutions to the agents' problems
and the aggregate investment on the left is the auctioneer's forecast. We refer to the
problem of ensuring that these equations hold as {\it the auctioneer's problem}. The
set of equations is determined: there is one unknown aggregate and one equation for
each non-terminal node in the tree of outcomes for the shocks. This is in the spirit of
a general equilibrium. Each node, even though it may not be realized, plays the role of
a good. The auctioneer must set the price so that demand matches supply. A solution of
the auctioneer's problem vindicates the agents' bounded rationality. They were right
to believe.

Each agent optimizes the expected discounted utility of his
consumption taking aggregate capital as given. We refer to this as {\it the agent's problem}. The agent's problem has a unique solution which is entirely determined by the agent's initial endowment together with the auctioneer's forecast.

One way of thinking of the interaction between the agent's problem and the auctioneer's problem is the following: one can consider a recursion in which the auctioneer fixes aggregates $K_t$ for all possible time periods and all possible outcomes of preceeding shocks.  The households then solve the agent's problem, optimizing expected discounted utility, taking the auctioneer's forecasts as given.  Each household $j$ thereby finds a unique solution consisting of random variables $\{k_{j,t},c_{j,t}\}$, which corresponds to its desired level of investment and corresponding consumption in period $t$ for each outcome of the shocks.  The auctioneer thereby obtains a listing of desired aggregates by adding together the values of $k_{j,t}$ across households.  If the desired aggregates are equal to the originally fixed aggregates, the auctioneer has solved her problem.  Otherwise, she will adjust her forecast and try again.

\subsection{Competitive Equilibrium} A competitive equilibrium is an allocation 
\begin{gather}
 \{(k_{j,t}, c_{j,t})_{j\in[1,N]}, K_t, L_t \}_{t=1}^T
\end{gather}
and a price path 
\begin{gather}
 \{R_{t}, w_{t} \}_{t=1}^T
\end{gather}
such that 
\begin{enumerate}
\item Given  $\{R_{t}, w_{t}, K_{t}\}_{t=1}^T$, the sequence $\{(k_{j,t}, c_{j,t})\}_{t=1}^T$ maximizes \eqref{eqn:Utility} subject to the constraints \eqref{eqn:BC1} for every $j$. 
\item Given  $\{R_{t}, w_{t} \}_{t=1}^T$, the sequence $\{K_t, L_t \}_{t=1}^T$ maximizes the firm's profit for every $t$. 
\item The capital and labor markets clear in each period. That is, $\sum_{j=1}^N k_{j,t} = K_t$,  $R_t = r_t-\delta$, and $\sum_{j=1}^N \xi_{j,t} = L_t$ for every $t$. 
\end{enumerate}

\subsection{Recursive Formulation}

In what follows, we impose complete depreciation of capital ($\delta=1$).  This assumption merely simplifies the language for the present discussion, and we will see below that we can repeat the same arguments with incomplete depreciation by reinterpreting the problem in terms of certain {\it effective variables}.

Let $\Omega_{t}$ denote the fraction of current output invested in period $t$ and let $s_{j,t}$ be household $j$'s share of that investment.  In the particular case of total depreciation ($\delta=1$), using the constancy of factor shares of the Cobb-Douglas production technology we can write household $j$'s budget constraint in terms of shares of current output $Y_t$,
\begin{gather}
c_{j,t} = [(1-\alpha) e_{j,t} + \alpha s_{j,t-1} - s_{j,t}\Omega_{t}]Y_t  \qquad 1\leq t\leq T  \label{eqn:BC}
\end{gather}

We denote the $j$th agent's initial share of aggregate wealth $Y_1$ by $\omega_{j,1}$:
\begin{gather}
\omega_{j,1} = (1-\alpha) e_{j,1} + \alpha s_{j,0},
\end{gather}
where we take $e_{j,1}$ and $s_{j,0}$ to be an initial endowment of labor and capital made known prior to any decision being made.  

When $T=1$, the solution of the agent's problem is straightforward.  There is no value gained by saving, so everything is consumed: $s_{j,1}=0$ for every $j$.  Using this fact, when $T=2$, the agent's problem can be written
\begin{gather}
 \max_{\{c_{t}\}_{t=1}^T} \left(\frac{c_1^{1-\sigma}}{1-\sigma}+\beta \E\frac{c_2^{1-\sigma}}{1-\sigma}\right)\notag\\
\text{subject to } \qquad c_{1}=[\omega_{j,1} - s_{j,1}\Omega_{1}]Y_{1},\quad  c_2 = [(1-\alpha) e_{j,2} + \alpha s_{j,1}]Y_2 \notag
\end{gather}
Substituting in the constraints, the problem becomes one of choosing the share $s_{j,1}$ of aggregate investment $\Omega_1 Y_1$ (as ordained by the auctioneer) that the agent will claim. Each agent will optimally chose his share of aggregate investment. Out of equilibrium, the share chosen by the agent might be subject to wishful thinking, in that such an agent solving his problem might come to the conclusion that he would like to hold more of the aggregate than the auctioneer has made available ($s_{j,1} > 1$). However, the auctioneer is bound by market clearing and such a situation indicates that equilibrium has not been achieved, insofar as the auctioneer's problem is not solved.

A solution $s_{j,1}=f_2(\omega_{j,1},e_{j,2},z_2,\Omega_{1},Y_1)$ of the above problem indicates a value function
\begin{gather}
 V_2 (\omega_{j,1};e_{j,2},z_2,\Omega_1,Y_1) = \frac{[(\omega_{j,1}-f_2(\omega_{j,1},e_{j,2},z_2,\Omega_{1},Y_{1})\Omega_{1})Y_1]^{1-\sigma}}{1-\sigma}\notag\\
\qquad\qquad\qquad\qquad\qquad\qquad\qquad\qquad+\qquad\beta \E\frac{[(\alpha f_2(\omega_{j,1},e_{j,2},z_2,\Omega_{1},Y_{1})+(1-\alpha)e_{j,2})Y_2]^{1-\sigma}}{1-\sigma}\notag
\end{gather}
Notice that $V_{2}$ contains a dependence on initial output, the fraction of this output held by the agent, the fraction of the aggregate to be invested, and the random variables representing the future shocks. 

When $T=3$, then, the agent seeks to maximize
\begin{gather}
 \frac{[(\omega_{j,1}-s_{j,1}\Omega_{1})Y_1]^{1-\sigma}}{1-\sigma}
+ \beta\E V_2 (\alpha s_{j,1} + (1-\alpha)e_{j,2};e_{j,3},z_3,\Omega_2,Y_2)\notag
\end{gather}
where the second term is the discounted expected future value given by the two period value function for an agent holding a share of output $Y_2$ as determined by investment returns $\alpha s_{j,1}$ plus employment earnings $(1-\alpha)e_{j,2}$ at the start of period 2.  

The solution of the above problem now yields a decision rule 
\begin{gather}
s_{j,1}=f_3(\omega_{j,1},\{e_{j,t}\}_{t=2}^3,\{z_t\}_{t=2}^3,\{\Omega_t\}_{t=1}^2,Y_1) 
\end{gather}
and corresponding value function 
\begin{gather}
 V_{3}(\omega_{j,1};\{e_{j,t}\}_{t=2}^3,\{z_t\}_{t=2}^3,\{\Omega_t\}_{t=1}^2,Y_1).
\end{gather}

Inducting on this procedure, in the case of a general (but finite) number of periods $T$, we obtain a value function
\begin{gather}
 V_{T}(\omega;\{e_{j,t}\}_{t=2}^T,\{z_t\}_{t=2}^T,\{\Omega_t\}_{t=1}^{T-1}, Y_1) =  \max_{s}\frac{((\omega_{j,1} - s_{j,1} \Omega_1)Y_1)^{1-\sigma}}{1-\sigma} \qquad\qquad\qquad\qquad\notag\\
 \qquad\qquad\qquad\qquad\qquad\qquad + \qquad \beta \E V_{T-1}[(1-\alpha) e_{j,2} + \alpha s_{j,1};\{e_{j,t}\}_{t=3}^T,\{z_t\}_{t=3}^T,\{\Omega_t\}_{t=2}^{T-1}, Y_2]. \label{eqn:RecursiveCRRA}
\end{gather}

An important simplification occurs in the case of $\log$ utility.  Namely, the agent's decisions are independent of the initial aggregate $Y_1$, so that we may reformulate the problem in such a way that the value function is also independent of this quantity. This result is reminiscent of the optimal portfolio literature in which agents with homothetic utility functions choose an allocation between risky and riskless assets that is independent of wealth level (see  \cite{Back2010asset}).

 Figure \ref{fig:Walrasian} depicts an event-tree view of the aggregate economy, while figure \ref{fig:agent} is the counterpart for the typical agent. As noted above and in the figures, variables dated $t$ are chosen or known within the period. At the beginning of the period the shocks are realized, then production occurs. The period concludes with the investment-consumption choice.     
 \begin{figure}[!h]
 \centering
  \begin{pspicture}[showgrid=false](11,4)
 \rput(2,2){\Rnode{A}{$Y_t$}}
 \rput(4,3){\Rnode{B}{$K_t$}}
 \rput(4,1){\Rnode{C}{$C_t$}}
 \rput(6.25,0.5){\Rnode{D}{$z_{t+1}$}}
 \rput(7.5,0.5){\Rnode{H}{$\xi_{t+1}$}}
 \rput(6.5,2){\Rnode{E}{$Y_{t+1} =z_{t+1} K_{t}^\alpha L_{t+1}^{1-\alpha}$}}
 \rput(10,3){\Rnode{F}{$K_{t+1}$}}
 \rput(10,1){\Rnode{G}{$C_{t+1}$}}
 \psset{nodesep=3pt,nrot=:U, arrows=->}
 \ncline{A}{B}
 \naput{$\Omega_t$}
 \ncline{A}{C}
 \nbput{$1-\Omega_t$}
 \ncline{B}{E}
 \ncline{D}{E}
 \ncline{H}{E}
 \ncline{E}{F}
 \naput{$\Omega_{t+1}$}
 \ncline{E}{G}
 \nbput{$1-\Omega_{t+1}$}
 \pcline{|-|}(0.5,0)(4.5,0)
 \lput*{:U}{$t$}
 \pcline{|-|}(5.5,0)(11,0)
 \lput*{:U}{$t+1$}
 \end{pspicture}
 \caption{Aggregate Economy}\label{fig:Walrasian}
 \vspace{0.25in}
  \begin{pspicture}[showgrid=false](11,4)
 \rput(2,2){\Rnode{A}{$\omega_t Y_t$}}
 \rput(4,3){\Rnode{B}{$k_t$}}
 \rput(4,1){\Rnode{C}{$c_t$}}
 \rput(6.25,0.5){\Rnode{D}{$z_{t+1}$}}
 \rput(7.35,0.5){\Rnode{H}{$e_{t+1}$}}
 \rput(6.5,2){\Rnode{E}{$( \alpha s_{t} + (1-\alpha )e_{t+1} )Y_t$}}
 \rput(11.5,3){\Rnode{F}{$k_{t+1}$}}
 \rput(11.5,1){\Rnode{G}{$c_{t+1}$}}
 \psset{nodesep=3pt,nrot=:U, arrows=->}
 \ncline{A}{B}
 \naput{$s_t\Omega_t/\omega_t$}
 \ncline{A}{C}
 \nbput{$1-s_t\Omega_t/\omega_t$}
 \ncline{B}{E}
 \ncline{D}{E}
 \ncline{H}{E}
 \ncline{E}{F}
 \naput{$s_{t+1}\Omega_{t+1}/\omega_{t+1}$}
 \ncline{E}{G}
 \nbput{$1-s_{t+1}\Omega_{t+1}/\omega_{t+1}$}
 \pcline{|-|}(0.5,0)(4.5,0)
 \lput*{:U}{$t$}
 \pcline{|-|}(5.5,0)(11,0)
 \lput*{:U}{$t+1$}
 \end{pspicture}
 \caption{Typical Agent's Problem ($\omega_t=\alpha s_{t}  +(1-\alpha) e_{t+1}$)} \label{fig:agent}
 \end{figure}
 
 Finally, we should note that we currently do not have a proof of existence and uniqueness for the auctioneer's problem. We submit the following conjecture:  
 
 \textbf{Conjecture 1}: \textit{For any number N of agents, number m of time periods, distribution
 of initial endowment and choice of employment shocks as described above, there exists a
 unique solution to the auctioneer's problem.}

 We unable as yet to resolve the conjecture but are working vigorously on it. What we
 present instead is a theorem about the extent to which this model approximately aggregates. Recall that approximate aggregation is a phenomenon first observed numerically in
 a very strong way by Krussel and Smith. What we do here is provide a rigorous justification of a fairly weak form of aggregation. We have as yet no real progress on obtaining
 rigorous existence and uniqueness of solutions to our model, since we can't solve the auctioneer's problem. 
 The reason that we find a rigorous description of approximate aggregation to be tractable
 is that it involves only the agent's problem, which is readily solved and whose solution is
 readily analyzed.
 
\section{Restrictions On The Allowed Shocks}

\subsection{Asymptotic Notation}

As a number of constants will enter our discussion, we will omit the constant by making use of asymptotic notation 
$A = \mathcal{O}(B).$   We introduce the asymptotic notation that if $A > cB$, we say that $A$ is $\mathcal{U}(B)$. (It is an underestimate instead of an overestimate.)

\subsection{Aggregate Shocks}
The aggregate shocks $z_t$ are allowed to be quite general.  In particular, the distribution of the shock in a given period is allowed to depend on the outcomes in prior periods.  However, we do require that the possible values of the aggregate shock be bounded:
\begin{gather}
 z_t = \mathcal{O}(1)
\end{gather}
with the implicit constant independent of the number of agents $N$ and the distribution of wealth.   

\subsection{Employment Shocks}
We place two restrictions on the employment shock, both of which are key in obtaining the aggregation result stated in the following section.

We impose the {\it risk of unemployment} condition, that the probability that $e_{j,t}=0$ is at least a constant $c>0$, also independent of $N$. Thus we are saying that, for each agent $j$ and period $t$,
\begin{gather}
P(e_{j,t}=0)=\mathcal{U}(1) \label{eqn:EmployProbRestriction}
\end{gather}
In other words, in every period, every agent has a chance of being unemployed.  This is important to ensure that agents must counteract a non-insureable risk of unemployment through their investment in each period.  Specifically, under this condition the asymptote in the utility function for zero consumption prevents agents from choosing to consume all that they have in non-terminal periods (under such a decision they face a positive probability of having nothing to consume next period, which results in a contribution of $-\infty$ to their discounted expected utility).  Likewise, this prevents borrowing against future wages.

The {\it similar future prospects} condition requires that the set $\{1,2, \dots, N\}$ is subdivided into a small number $s$ of subsets
$A_1, \dots, A_s$ so that for any $j,k \in A_l$, for some $l$, we have that $e_{j,t}$ and $e_{k,t}$ have the same distribution at each time $t$.  The need for this comes, simply, from the fact that agents facing a very disimilar prospect of unemployment or underemployment cannot be guaranteed to invest a similar amount, even in the case that they have the same current wealth.  Such agents can still be aggregated, provided that there are not so many disimilarities that they overwhelm the estimates.

We conclude this section by providing some basic examples of shocks one might consider.

\subsection{Examples}



{\bf Constant Wages and Uniform Employment} Suppose that in each period $(1-u)N$ of the agents are chosen, for employment, and each is paid an equal fraction $1/((1-u)N)$ of the available wages.  The remaining $uN$ agents go unemployed.  Clearly there is a risk of unemployment (a fraction $uN$ randomly face unemployment), and there are similar future prospects (everyone faces the same chance of employment and unemployment, every period).

Hence we can conclude approximate aggregation in this model via the theorem below. However, the aggregation is not perfect; even for two periods there now appears an error term which is of size $O(1/N)$.  Notice that this is negligible for the very rich, but is worth consideration for the very poor.  

{\bf  Krussell-Smith Shocks.}  Suppose that the technology process follows a Markov chain with two states, $z_g>z_b$ (the `good' state and the `bad' state) with transition probabilities $p_{xy}$, $x,y\in \{g,b\}$.  Further, as above, agents are either employed (1) or unemployed (0), and wages are distributed evenly among employed agents.  However, in this instance, the chances of transitioning between the various states of employment differ, and moreover may depend on the state of the aggregate shock.  Specifically, let $\pi_{ss'ee'}$ denote the joint probability of transition from state $(z_s,e)$ to state $(z_{s'},e')$.  Then the ratio $\pi_{ss'ee'}/\pi_{ss'}$ is the conditional probability of of transitioning into employment status $e'$ from $e$.  

For this stochastic process, there is a risk of unemployment as long as $\pi_{ss'e0}\neq 0$ for any $s,s'\in\{g,b\}$ and $e\in\{0,1\}$. Future prospects are also bounded: agents who are initially employed ($A_1$) may differ in their decisions from those who are initially unemployed ($A_2$) if the chances of transition between different employment states are very small; however everyone falls into one of these groups regardless of the size of $N$.
 
\subsection{Effective Variables and Undepreciated Capital} 

We now return to the recursive formulation (\ref{eqn:RecursiveCRRA}), which we wrote down in the case of total depreciation of capital.  We wish to be able to generalize this to the case where capital does not depreciate fully.  

To this end, we define the effective aggregate $Y_t'$, the total number of goods available:
\begin{gather}
Y_t ' := Y_t + \frac{(1-\delta)}{\alpha}\Omega_{t-1} Y_{t-1} > Y_t.
\end{gather}
We also rewrite the labor shock as an effective shock which is measured relative to this effective aggregate:
\begin{gather}
e_{j,t}' := e_{j,t} \frac{Y_t}{Y_t'} < e_{j,t}. 
\end{gather}
In particular, note that the effective employment shock satisfies both of the conditions of the previous discussion provided the real shock does.  In practice, it can actually become significantly smaller; the factor used to obtain it from the real shock can be rewritten as
\begin{gather}
 \frac{1}{1+\frac{(1-\delta)}{\alpha}\Omega_{t-1} Y_t^{1-\alpha}}.
\end{gather}
In the case of total depreciation this is just 1, and we are left with the original real shock.  However, when $\delta<1$ and $\alpha<1$, a sufficiently large aggregate can make this expression arbitrarily small: an agent is able to consume primarily out of undepreciated capital, and employment becomes a secondary concern.

In terms of the effective variables, we can now follow the previous procedure to write the recursive formulation
\begin{gather}
 V_{T}(\omega;\{e_{j,t}'\}_{t=2}^T,\{z_t'\}_{t=2}^T,\{\Omega_t\}_{t=1}^{T-1}, Y_1) =  \max_{s}\frac{((\omega - s \Omega_1)Y_1)^{1-\sigma}}{1-\sigma} \qquad\qquad\qquad\qquad\notag\\
 \qquad\qquad\qquad\qquad\qquad\qquad + \qquad \beta \E V_{T-1}[(1-\alpha) e_{2}' + \alpha s;\{e_{j,t}'\}_{t=3}^T,\{z_t'\}_{t=3}^T,\{\Omega_t\}_{t=2}^{T-1}, Y_2'].
\end{gather}
Seeing that the only change is the direct replacement of real variables by effective ones, in the discussion that follows we will suppress the $'$ notation, while simply remembering that all aggregate and labor variables refer to their effective counterparts.

\section{Approximate Aggregation}

\subsection{Statements of the Main Theorems}

We now fix the following quantities: the number of time periods $T$, the discount factor $\beta$, the Cobb-Douglas exponent $\alpha$.  Let $N$ denote the number of agents, which we think of as being large compared to $T$.  Let $\{z_{t},e_{j,t}\}$ be the aggregate and employment shocks, and let $\Omega_{t}$, $t=1,...,T-1$ be a sequence of forecasts given by the auctioneer.  Consider the corresponding agent's problem (for the $j$th agent) given by 
\begin{gather}
\max_{s}\frac{((\omega - s \Omega_1)Y_1)^{1-\sigma}}{1-\sigma} \qquad\qquad\qquad\qquad\notag\\
 \qquad\qquad\qquad\qquad\qquad\qquad + \qquad \beta \E V_{T-1}[(1-\alpha) e_2 + \alpha s;\{e_{j,t}\}_{t=3}^T,\{z_t\}_{t=3}^T,\{\Omega_t\}_{t=2}^{T-1}, Y_2],\notag
\end{gather}
where, as noted at the end of the last section, $e_{j,t}$ and $Y_t$ are the effective employment and total number of goods available, respectively.

It is convenient to analyze to present the aggregation result in terms of the {\it savings variable}, $\gamma = s \Omega_1/\omega$, which represents the ratio of the agent's desired period 1 savings $s\Omega_1 Y_1$ to that agent's initial wealth $\omega Y_1$. Thus it is convenient to write the agent's problem as a choice over $\gamma$, 
\begin{gather}
 \max_{\gamma}\frac{((1-\gamma)\omega Y_1)^{1-\sigma}}{1-\sigma} + \beta \E V_{T-1}\left[\frac{\alpha\omega\gamma}{\Omega_1} + (1-\alpha) e_2;\{e_{j,t}\}_{t=3}^T,\{z_t\}_{t=3}^T,\{\Omega_t\}_{t=2}^{T-1}, Y_2\right]. \label{eqn:RecursiveLambda}
\end{gather} 
Notice that a solution of the agent's problem suggests a decision rule for $\gamma$, 
\begin{gather}
\gamma_{T}(\omega;\{e_{j,t}\}_{t=2}^T,\{z_t\}_{t=2}^T,\{\Omega_t\}_{t=1}^{T-1},Y_{1}). \notag
\end{gather}
We refer to this decision rule as the {\it savings function}. 

For a subset $B\subset \{1,...,N\}$, we let 
$$
Y_{l,1} = \sum_{j\in B} \omega_{j,1} Y_1.
$$
be the portion of the aggregate wealth held by the agents who are indexed by $B$.  In this notation, our main theorem can be written as follows.

\begin{thm}\label{thm:ApproximateAggregation} Suppose that the sequence of production and employment shocks $\{z_{t},e_{j,t}\}$ satisfy the risk of unemployment, and similar future prospects conditions.  Let $\epsilon>0$.  Then there is a natural number $M=\mathcal{O}(1/\epsilon)$ such that we can partition $\{1,...,N\}$ into subsets $B_1,...,B_M$ with corresponding ratios $\gamma_{1},...,\gamma_{M}\in (0,1)$ such that
$$
|\sum_{j\in\{1,...,N\}} \gamma_{j,1} \omega_{j,1}Y_1 - \sum_{m\in\{1,...,M\}} \gamma_{m} Y_{m,1}| \leq \epsilon Y_1.
$$
\end{thm}

To understand the relation of this theorem to aggregation, it is instructive to look at what occurs if we were able to take $M=1$ (this is not possible in practice outside of the trivial case when $\epsilon>1$).   In that case, the theorem says that if we lump all agents together into a single entity $B_1$ holding wealth $\sum_{j\in B_1} \omega_{j,1}Y_1=Y_1$ and saving a fraction $\gamma_1$ of its wealth, the error between the aggregate investment thus computed and the true aggregate initial period investment $\sum_{j} \gamma_{j,1} \omega_{j,1}Y_1$ will be at least as small as $2\epsilon Y_1$.  In other words, all agents combined would save $\gamma Y_1$ where $\gamma$ may change as we alter the distribution of wealth among agents, but nevertheless will remain in the interval $[\gamma_1-2\epsilon,\gamma_1+2\epsilon]$.

For meaningful values of $\epsilon$, the above theorem is slightly weaker: it says that we can redistribute wealth among agents within any single one of $\mathcal{O}(1/\epsilon)$ bins without changing the aggregate substantially.  In other words, we have {\it approximate aggregation}.  

A key feature of the above theorem is that the number of bins is independent of both the number of agents and the distribution of wealth among those agents.  The intuition which motivates our isolating of $N$ and the distribution as the dependencies of interest is the potential for using the limit $N\to\infty$ to approximate continuous distribution models such that outlined in KS.  As such, the above aggregation result is constructed to be robust no matter how large we might take $N$ to be. 

The main theorem will follow, by a pigeonholing argument, as a corollary of the following theorem investigating the behavior of the derivative of the savings function with respect to agent wealth.

\begin{thm}\label{thm:GammaDerivative} Suppose that the sequence of production and employment shocks $\{z_t,e_{j,t}\}$ satisfy the risk of unemployment and similar future prospects conditions.  Let $\gamma_T'$ denote the partial derivative of the function $\gamma_T$ with respect to its first variable (the share variable, $\omega$) and similarly for $\gamma_{T-1}'$.

Then $\gamma_T$ is increasing and we have a bound on $\gamma_T'(\omega;\{e_{j,t}\}_{t=2}^T,\{z_t\}_{t=2}^T,\{\Omega_t\}_{t=1}^{T-1},Y_{1})$ given by
\begin{gather}
\mathcal{O}\left(\frac{1}{\omega^{1-\sigma}}\E\left[\frac{e_{j.2}}{y_{j,2}^2}+y_{j,2}^{1-\sigma}\gamma_{T-1}'(y_{j,2};\{e_{j,t}\}_{t=3}^T,\{z_t\}_{t=3}^T,\{\Omega_t\}_{t=2}^{T-1},Y_{2})\right]\right).\notag
\end{gather}
where the random variable $y_{j,2}$ is given by $\alpha\omega\gamma_T/\Omega_1+(1-\alpha)e_2$.
\end{thm}

\subsection{Lemmas Regarding Value and Savings Functions}

The next two sections will contain the large majority of our technical arguments.  For this portion of the discussion we will simplify the notation in order to focus on the most relevant dependencies of our value functions and decision rules.  

Specifically, we will write
\begin{gather}
 V_T(\omega; Y_1)\ {\rm for\ } V_{T}(\omega;\{e_{j,t}\}_{t=2}^T,\{z_t\}_{t=2}^T,\{\Omega_t\}_{t=1}^{T-1},Y_{1})\notag\\
 s_T(\omega; Y_1)\ {\rm for\ } f_{T}(\omega;\{e_{j,t}\}_{t=2}^T,\{z_t\}_{t=2}^T,\{\Omega_t\}_{t=1}^{T-1},Y_{1})\notag\\
 \gamma_T(\omega; Y_1)\ {\rm for\ } \gamma_{T}(\omega;\{e_{j,t}\}_{t=2}^T,\{z_t\}_{t=2}^T,\{\Omega_t\}_{t=1}^{T-1},Y_{1})\notag
\end{gather}
All derivatives will be with respect to the first variable, $\omega$, and will be denoted (respectively) by $V_T'(\omega; Y_1)$, $s_T'(\omega; Y_1)$, and $\gamma_T'(\omega; Y_1)$, with analogous expressions for higher order derivatives. We assume throughout that the shocks $\{z_t,e_{j,t}\}$ satisfy the conditions laid out previously.  

We further simplify the notation by speaking about the problem faced by a \emph{typical agent}, and correspondingly suppressing the subscripts $1\leq j\leq N$ on employment shocks.

In the following arguments, we will frequently need to refer to the first order condition for the agent's problem.  For convenience we record it here, once and for all:
\begin{gather}
 \frac{Y_1^{1-\sigma}}{(\omega-s\Omega_1)^{\sigma}} = \frac{\beta\alpha}{\Omega_1} \E V_{T-1}'((\alpha s +(1-\alpha) e_{2});Y_{2})\label{eqn:FirstOrderCondition1}
\end{gather}
and, in terms of the savings variable $\gamma$,
\begin{gather}
 \frac{Y_1^{1-\sigma}\omega^{1-\sigma}}{(1-\gamma)^\sigma} = \frac{\beta\alpha\omega}{\Omega_1}\E V_{T-1}'\left(\frac{\alpha\omega\gamma}{\Omega_1} + (1-\alpha) e_2;Y_{2}\right).\label{eqn:FirstOrderCondition2}
\end{gather}

We begin by demonstrating that the agent's bounded rationality makes it easy to solve the agent's problem.

\begin{lemma} 
There is a unique solution to the agent's problem.  The decision rule is increasing with respect to $\omega$, and the corresponding value function is strictly increasing and strictly concave with respect to $\omega$.  Moreover
\begin{gather}
 \lim_{\omega\to 0^+} V_T(\omega;Y_1) = -\infty
\end{gather}

 \end{lemma}
\begin{proof}
 When $T=1$, this is clear: the decision rule is to consume everything. The value function is $u(\omega; Y_1)$, which is clearly increasing and strictly concave as a function of $\omega$, with the given asymptote.

Let $T>1$, and suppose we have a unique solution to the $T-1$ period model with a value function having the stated properties.  We need to show that, for a given value of $\omega$, there is a unique value of $s\in (0,\omega/\Omega_1)$ satisfying (\ref{eqn:FirstOrderCondition1}).  The left hand side of this equation is strictly increasing in $s$ from $\frac{Y_1^{1-\sigma}}{\omega^{\sigma}}>0$ to $\infty$.  As $s\to 0^+$, the right side approaches $\infty$, due to positive probability that $e_2=0$ (possibility of unemployment); moreover it is strictly decreasing by concavity of $V_{T-1}$.  From these facts it is clear that such a value of $s$ must exist, and moreover must be unique.

That $s_T$ is increasing in $\omega$ can be seen from an inspection of (\ref{eqn:FirstOrderCondition1}).  The fact that $V_{T-1}$ is increasing will be apparent from \ref{lemma:ValueFunctionDerivatives} below, and we see that the asymptote is inherited from $V_{T-1}$ due to the possibility of unemployment.  We must therefore only show that the $T$ period value function is concave.  

In our simplified notation, we have
\begin{gather}
 V_T(\omega; Y_1) = \frac{((\omega - s_{T}(\omega;Y_1) \Omega_1)Y_1)^{1-\sigma}}{1-\sigma} \qquad\qquad\qquad\qquad\notag\\
 \qquad\qquad\qquad\qquad\qquad\qquad + \qquad \beta \E V_{T-1}[((1-\alpha) e_2 + \alpha s_{T}(\omega;Y_1)); Y_2].\notag
\end{gather}
Differentiating with respect to $\omega$ we have
\begin{gather}
V_T'(\omega;Y_1) = \frac{Y_{1}^{1-\sigma}(1- s_{T}'(\omega;Y_1)\Omega_1)}{(\omega - s_{T}(\omega;Y_1) \Omega_1)^\sigma} \qquad\qquad\qquad\qquad\notag\\
 + \qquad \beta \E  \alpha s_{T}'(\omega;Y_1)V_{T-1}'[((1-\alpha) e_2 + \alpha s_{T}(\omega;Y_1)); Y_2],\notag
\end{gather}
and doing so again we have
\begin{gather}
 V_T''(\omega;Y_1) = \frac{-\sigma Y_{1}^{1-\sigma}(1- s_{T}'(\omega;Y_1)\Omega_1)^2}{(\omega - s_{T}(\omega;Y_1) \Omega_1)^{1+\sigma}} + \frac{Y_{1}^{1-\sigma}(- s_{T}''(\omega;Y_1)\Omega_1)}{(\omega - s_{T}(\omega;Y_1) \Omega_1)^\sigma} 
\qquad\qquad\qquad\qquad\notag\\ 
+ \qquad  \beta \E  (\alpha s_{T}'(\omega;Y_1))^2 V_{T-1}''[((1-\alpha) e_2 + \alpha s_{T}(\omega;Y_1)); Y_2] 
\qquad\qquad\qquad\qquad\notag\\ 
+ \qquad \beta \E  \alpha s_{T}''(\omega;Y_1)V_{T-1}'[((1-\alpha) e_2 + \alpha s_{T}(\omega;Y_1)); Y_2].\notag
\end{gather}
Invoking the first order condition, we see that the terms involving $s_{T}''$ cancel, leaving
\begin{gather}
 V_T''(\omega;Y_1) = \frac{-\sigma Y_{1}^{1-\sigma}(1- s_{T}'(\omega;Y_1)\Omega_1)^2}{(\omega - s_{T}(\omega;Y_1) \Omega_1)^{1+\sigma}}
\qquad\qquad\qquad\qquad\notag\\ 
+ \qquad  \beta \E  (\alpha s_{T}'(\omega;Y_1))^2 V_{T-1}''[((1-\alpha) e_2 + \alpha s_{T}(\omega;Y_1)); Y_2]. \notag
\end{gather}
Both remaining terms are negative: this is clear by investigation for the first term, and follows from concavity of the one-period-less value function in the second.  This closes the induction.
\end{proof}

When analyzing the derivative of $\gamma_T(\omega;Y_1)$, we will be forced to confront derivatives of the value function $V_{T-1}(\omega;Y_1)$ emergent in the first order conditions.  We will need the formulae

\begin{lemma}\label{lemma:ValueFunctionDerivatives}
 \begin{gather}
 V_{T}'(\omega;Y_1) = \frac{Y_1^{1-\sigma}}{(\omega-s_{T}(\omega;Y_1)\Omega_1)^\sigma}\notag\\
 V_{T}''(\omega;Y_1) = \frac{\sigma Y_1^{1-\sigma}(s_{T}'(\omega;Y_1)\Omega_1 - 1)}{(\omega-s_T(\omega;Y_1)\Omega_1)^{1+\sigma}}.\notag
 \end{gather}
\end{lemma}
We remark that in the $\log$ utility case, the derivatives of the value function become the limits of the above expressions as $\sigma\to 1$.  This can be seen in an almost identical (albeit simpler) calculation to the below. 

\begin{proof}
The value function is what we obtain by substituting the optimal decision rule into the agent's discounted expected consumption.  Differentiating the resulting function with respect to $\omega$, we get
\begin{gather}
 V_T'(\omega; Y_1) = \frac{Y_1^{1-\sigma}(1-s_T'(\omega;Y_1)\Omega_1)}{(\omega-s_T(\omega;Y_1)\Omega_1)^\sigma} + \beta\alpha s_T'(\omega;Y_1)\E  Y_2 V_{T-1}'(\alpha s_T(\omega;Y_1) + (1-\alpha)e;Y_2).\notag
\end{gather}
Using the first order condition (\ref{eqn:FirstOrderCondition1}), the terms involving $s_T'$ cancel, and we get
\begin{gather}
 V_T'(\omega; Y_1) = \frac{Y_1^{1-\sigma}}{(\omega-s_T(\omega;Y_1)\Omega_1)^\sigma}\notag
\end{gather}
Differentiating the above expression for $V'$, we get
\begin{gather}
 V_T''(\omega;Y_1) = \frac{\sigma Y_1^{1-\sigma}(s_T'(\omega;Y_1)\Omega_1 - 1)}{(\omega-s_T(\omega;Y_1)\Omega_1)^{1+\sigma}}.\notag
\end{gather}
\end{proof}

Next, in order to bound the derivative of $\gamma$ with respect to $\omega$, we will perform an analysis which will be aided by our ability to write $\mathcal{U}(1)$ and $\mathcal{O}(1)$ for expressions dependent on model parameters, allowing us to focus on the variables of interest.  Several of these expressions will involve factors of $\gamma$ or $1-\gamma$, and thus we provide the following sequence of technical lemmas, the net result of which is a statement slightly stronger than the fact that $0<\gamma_T(\omega;Y_1)<1$.  

We begin by observing that an agent who faces guaranteed unemployment will invest a larger fraction of his income than one who has a potential opportunity to obtain wages.  Hence we can estimate
\begin{gather}
 1-\gamma_T(\omega;Y_1) > 1-\overline{\gamma}_T(\omega; Y_1)\notag
\end{gather}
where $\overline{\gamma}_T(\omega; Y_1)$ is the savings function corresponding to the agent's problem with all employment shocks equal to 0.  It will be apparent in the proof below that the right hand side is independent of $\omega$, as should be expected due to the lack of heterogeneity in the no-employment extreme case.  We can therefore write
\begin{gather}
 1-\gamma_T(\omega;Y_1) > 1-\overline{\gamma}_T(\{Y_t\}_{t=2}^T, Y_1).\notag
\end{gather}

\begin{lemma}\label{lemma:GammaUB} We have
\begin{gather}
 \frac{1}{1-\overline{\gamma}_2(\{Y_2\}, Y_1)} = 1+\left(\beta \alpha^{1-\sigma}\E \frac{Y_{2}^{1-\sigma}}{(\Omega_{1} Y_{1})^{1-\sigma}}\right)^{1/\sigma}\notag
\end{gather}
and, for $T>2$, we have the recursive expression
\begin{gather}
 \frac{1}{1-\overline{\gamma}_T(\{Y_t\}_{t=2}^T,Y_1)} := 1+\left(\beta\alpha^{1-\sigma}\E \frac{Y_{2}^{1-\sigma}}{(\Omega_{1} Y_{1})^{1-\sigma}} \left(\frac{1}{(1-\overline{\gamma}_{T-1}(\{Y_t\}_{t=3}^T,Y_2))^\sigma}\right) \right)^{1/\sigma}.\label{eqn:GammaBound1}
\end{gather}
\end{lemma}
A remark before the proof: Notice that in the case of $\log$ utility, the inequality given by the lemma reduces to the significantly simpler expression
\begin{gather}
 1-\overline{\gamma_T} = \frac{1}{1+\beta+...+\beta^{T-1}}.\notag
\end{gather}

\begin{proof}
The proof is just a rewriting of first order conditions.  We proceed by induction.  

When $T=2$, the agent maximizes
\begin{gather}
 \frac{Y_1^{1-\sigma}\omega^{1-\sigma}(1-\overline{\gamma})^{1-\sigma}}{1-\sigma} + \beta \E \frac{\frac{\alpha\omega}{\Omega_1}\overline{\gamma}^{1-\sigma}Y_2^{1-\sigma}}{1-\sigma}\notag
\end{gather}
Taking first order conditions with respect to $\gamma$ results in
\begin{gather}
 \frac{Y_{1}^{1-\sigma}\omega^{1-\sigma}}{(1-\overline{\gamma})^\sigma} = \beta\left(\frac{\alpha\omega}{\Omega_1}\right)^{1-\sigma} \frac{1}{\overline{\gamma}^\sigma} \E Y_2^{1-\sigma}.\notag
\end{gather}
After some cancelling the factors involving $\omega$, taking $\sigma$th roots, and rearranging, this gives
\begin{gather}
 \frac{1}{1 - \overline{\gamma}} = 1+\left(\beta \alpha^{1-\sigma}\E \frac{Y_2^{1-\sigma}}{(\Omega_1 Y_{1})^{1-\sigma}}\right)^{1/\sigma}\notag
\end{gather}
which is the estimate for $T=2$. 

The inductive step proceeds similarly.  Suppose that we have demonstrated that $1-\overline{\gamma}_{T-1}$ can be expressed as the right hand side of (\ref{eqn:GammaBound1}).  The agent for the $T$ period case solves
\begin{gather}
 \frac{Y_1^{1-\sigma}\omega^{1-\sigma}(1-\overline{\gamma})^{1-\sigma}}{1-\sigma} + \beta \E V_{T-1}\left(\frac{\alpha\omega\overline{\gamma}}{\Omega_{1}};Y_2\right).\notag
\end{gather}
Applying Lemma \ref{lemma:ValueFunctionDerivatives}, we can write the first order conditions (in terms of $\overline{\gamma}$) as
\begin{gather}
 \frac{Y_{1}^{1-\sigma}\omega^{1-\sigma}}{(1-\overline{\gamma})^\sigma} = \beta \left(\frac{\alpha\omega}{\Omega_{1}}\right)^{1-\sigma}\frac{1}{\overline{\gamma}^\sigma}\E \frac{Y_2^{1-\sigma}}{(1-\overline{\gamma}_{T-1}(\{Y_t\}_{t=3}^T;Y_2))^\sigma}. \notag
\end{gather}
Rearranging, we have
\begin{gather}
 \left(\frac{\overline{\gamma}}{1-\overline{\gamma}}\right)^\sigma = \beta \alpha^{1-\sigma} \E \left(\frac{Y_2^{1-\sigma}}{(\Omega_{1} Y_{1})^{1-\sigma}}\left(\frac{1}{(1-\overline{\gamma}_{T-1}(\{Y_t\}_{t=3}^T;Y_2))^\sigma}\right)\right),\notag
\end{gather}
and applying the inductive hypothesis to the latter factor in each term of the expected value we are done.
\end{proof}

A sibling to the above result is the following, which will allow us to write $\mathcal{U}(1)$ for $\gamma_T(\omega,Y_1)$ itself.  We write $\E_{e_2=0} X$ for the expected value of the random variable $X$ over the event that $e_2=0$.  That is, 
$$
\E_{e_2=0} X = \E X 1_{e_2=0}
$$
where $1_A$ denotes the characteristic function of the event $A$.

\begin{lemma}\label{lemma:GammaLB} Let 
\begin{gather}
 \frac{1}{1-\underline{\gamma}_2(\{Y_2\},Y_1)} := 1+\left(\beta \alpha^{1-\sigma}\E_{e_2=0} \frac{ Y_{2}^{1-\sigma}}{(\Omega_1 Y_1)^{1-\sigma}}\right)^{1/\sigma}\notag
\end{gather}
and for $T > 2$ let
\begin{gather}
 \frac{1}{1-\underline{\gamma}_T(\{Y_t\}_{t=2}^T,Y_1)} := 1+\left(\beta\alpha^{1-\sigma}\E_{e_2=0} \frac{ Y_{2}^{1-\sigma}}{(\Omega_1 Y_1)^{1-\sigma}} \left(\frac{1}{(1-\underline{\gamma}_{T-1}(\{Y_t\}_{t=3}^T,Y_2)^\sigma}\right) \right)^{1/\sigma}\label{eqn:GammaBound2}
\end{gather}
Then
\begin{gather}
1-\gamma_T(\omega;Y_1) \leq 1-\underline{\gamma}_T(\{Y_t\}_{t=2}^T, Y_1).\notag
\end{gather}
\end{lemma}

\begin{proof}
 In the first order condition (\ref{eqn:FirstOrderCondition2}), taking only those terms with $e_2 = 0$ we get
\begin{gather}
 \frac{Y_1^{1-\sigma}\omega^{1-\sigma}}{(1-\gamma)^\sigma} \geq \frac{\beta\alpha\omega}{\Omega_1}\E_{e_2 = 0}  V_{T-1}'\left(\frac{\alpha\omega\gamma}{\Omega_1};Y_{2}\right).\notag
\end{gather}
Applying Lemma \ref{lemma:ValueFunctionDerivatives} gives
\begin{gather}
 \frac{Y_{1}^{1-\sigma}\omega^{1-\sigma}}{(1-\gamma)^\sigma} \geq \beta \left(\frac{\alpha\omega}{\Omega_{1}}\right)^{1-\sigma}\frac{1}{\gamma^\sigma}\E_{e_2 = 0} \frac{ Y_2^{1-\sigma}}{(1-\gamma_{T-1}(\frac{\alpha\omega\gamma}{\Omega_1};Y_2))^\sigma}.\notag
\end{gather}
In the case $T=2$, $\gamma_{T-1} = \gamma_1 = 0$, and we can rearrange this as in the previous lemma.  Inducting, then, and assuming the result for $T-1$ periods, we can bound $1-\gamma_{T-1}$ from above by $1-\underline{\gamma}_{T-1}$ giving
\begin{gather}
 \frac{Y_{1}^{1-\sigma}\omega^{1-\sigma}}{(1-\gamma)^\sigma} \geq \beta \left(\frac{\alpha\omega}{\Omega_{1}}\right)^{1-\sigma}\frac{1}{\gamma^\sigma}\E_{e_2 = 0} \frac{Y_2^{1-\sigma}}{(1-\underline{\gamma}_{T-1}(\{Y_t\}_{t=3}^T,Y_2))^\sigma}.\notag
\end{gather}
Rearranging the last expression gives the general result.
\end{proof}

In practice, we will not directly apply the above lemmas, but instead refer to the following bounds, which follow from a rearrangement of terms in the previous two lemmas.

\begin{lemma} We have
\begin{gather}
1-\gamma_T(\omega,Y_1) = \mathcal{U}(1)\notag\\
\gamma_T(\omega, Y_1) = \mathcal{U}(1)\notag
\end{gather}
where the implicit constants can be made independent of $N$ and $\omega$.
\end{lemma}

\subsection{Derivative of the Savings Function}

We are now ready to prove \ref{thm:GammaDerivative}.   Let $T\geq 2$ (there is nothing to prove in the $T=1$ case).  We rearrange (\ref{eqn:FirstOrderCondition2}) as
\begin{gather}
 \frac{1}{(1-\gamma)^\sigma} = \frac{\beta\alpha\omega^\sigma}{\Omega_1}\frac{1}{Y_1^{1-\sigma}}\E V_{T-1}'\left(\frac{\alpha\omega\gamma}{\Omega_1} + (1-\alpha) e_2; Y_2\right).\notag
\end{gather}
As before, this condition implies a maximizing savings function $\gamma_T = \gamma_T(\omega;Y_1)$.

Our first step is to differentiate the last line implicitly with respect to $\omega$.  We get
\begin{gather} \begin{aligned}
 \frac{\sigma\gamma_T'}{(1-\gamma_T)^{1+\sigma}} = \frac{\beta\alpha\sigma\omega^{\sigma-1}}{\Omega_1}\frac{1}{Y_1^{1-\sigma}}\E  V_{T-1}'\left(\frac{\alpha\omega\gamma_T}{\Omega_1}+ (1-\alpha) e_2;Y_{2}\right)\ \ \ \ \ \ \ \ \ \ \ \ \ \ \ \ \notag\\+\frac{\beta\alpha\omega^\sigma}{\Omega_1}\frac{1}{Y_1^{1-\sigma}}\left[\frac{\alpha\gamma_T}{\Omega_1}+\frac{\alpha\omega\gamma_T'}{\Omega_1}\right]\E  V_{T-1}''\left(\frac{\alpha\omega\gamma_T}{\Omega_1} + (1-\alpha) e_2;Y_{2}\right).\notag
\end{aligned} \end{gather}
Solving for $\gamma_T'$, this gives
\begin{gather}
 \gamma_T'(\omega) = \frac{\frac{\beta\alpha\omega^{\sigma-1}}{\Omega_1}\frac{1}{Y_1^{1-\sigma}} \left[\E  \left(\sigma V_{T-1}'\left(\frac{\alpha\omega\gamma_T}{\Omega_1} + (1-\alpha) e_2; Y_{2}\right)+\frac{\alpha\omega\gamma_T}{\Omega_1}V_{T-1}''\left(\frac{\alpha\omega\gamma_T}{\Omega_1} + (1-\alpha) e_2;Y_{2}\right)\right)\right]}{\left[\frac{\sigma}{(1-\gamma_T)^{1+\sigma}} - \frac{\beta\alpha^2}{\Omega_1^2}\omega^{1+\sigma}\frac{1}{Y_1^{1-\sigma}}\E  V_{T-1}''\left(\frac{\alpha\omega\gamma_T}{\Omega_1} + (1-\alpha) e_2;Y_{2}\right)\right]}\label{eqn:GammaDerivativeUnsimplified}
\end{gather}

We can simplify the numerator a fair bit.  We let $x=\frac{\alpha\omega\gamma_T}{\Omega_1}$ and rewrite it as
\begin{gather}
 \frac{\beta\alpha}{\Omega_1 \omega^{1-\sigma}Y_1^{1-\sigma}} \E  \left(\sigma V_{T-1}'\left(x + (1-\alpha) e_2\right)+xV_{T-1}''\left(x + (1-\alpha) e_2\right)\right)\label{eqn:SavingsDerivativeNumerator}
\end{gather}
Introducing a random variable $y_2=x+(1-\alpha)e_2$ and rewriting a typical term of the last expression using Lemma \ref{lemma:ValueFunctionDerivatives} gives
\begin{gather}
\frac{\beta\alpha}{\Omega_1 \omega^{1-\sigma}Y_1^{1-\sigma}}\left[\frac{\sigma Y_{2}^{1-\sigma}}{(x+(1-\alpha)e_2 - s_{T-1}(y_2)\Omega_{2})^\sigma} + \frac{\sigma xY_{2}^{1-\sigma}(s_{T-1}'(y_2)\Omega_{2}-1)}{(x+(1-\alpha)e_2 - s_{T-1}(y_2)\Omega_{2})^{1+\sigma}}\right]\notag
\end{gather}
with $s_{T-1}$ being shorthand for the $T-1$ period decision rule, $s_{T-1}(y_2)=s_{T-1}(y_2;Y_2)$.  This simplifies further to become
\begin{gather}
\frac{\beta\alpha}{\Omega_1 \omega^{1-\sigma}Y_1^{1-\sigma}}\sigma Y_{2}^{1-\sigma}\left[\frac{(1-\alpha)e_2 - s_{T-1}(y_2)\Omega_{2} + xs_{T-1}'(y_2)\Omega_{2}}{(x+(1-\alpha)e_2 - s_{T-1}(y_2)\Omega_{2})^{1+\sigma}}\right].\label{eqn:SavingsDerivativeNumerator2}
\end{gather}

To make sense of the expression $- s_{T-1}(y_2)\Omega_{2} + xs_{T-1}'(y_2)\Omega_{2}$, we note that
\begin{gather}
 - s_{T-1}(y_2)\Omega_{2} + xs_{T-1}'(y_2)\Omega_{2} = - s_{T-1}(y_2)\Omega_{2} + y_2 s_{T-1}'(y_2)\Omega_{2} - (1-\alpha)e_2 s_{T-1}'(y_2)\Omega_{2}.\notag
\end{gather}
However, we also have the definition
\begin{gather}
 \gamma_{T-1}(y_2) = \frac{s_{T-1}(y_2)\Omega_{2}}{y_2}\notag
\end{gather}
and so differentiating with respect to $y_2$ we get
\begin{gather}
 \gamma_{T-1}'(y_2) = \frac{s_{T-1}'(y_2)\Omega_2}{y_2} - \frac{s_{T-1}(y_2)\Omega_2}{y_2^2}.\notag
\end{gather}
whereby
\begin{gather}
 -s_{T-1}(y_2)\Omega_2 + s_{T-1}'(y_2)y_2\Omega_2 = y_2^2\gamma_{T-1}'(y_2).\notag
\end{gather}
and so
\begin{gather}
 - s_{T-1}(y_2)\Omega_2 + xs_{T-1}'(y_2)\Omega_2 = y_2^2\gamma_{T-1}'(y_2) - (1-\alpha)e_2 s_{T-1}'(y_2)\Omega_2.\label{eqn:SavingsDerivativeNumerator3}
\end{gather}

Combining equations (\ref{eqn:SavingsDerivativeNumerator2}) and (\ref{eqn:SavingsDerivativeNumerator3}), we see that (\ref{eqn:SavingsDerivativeNumerator}) is equal to
\begin{gather}
\frac{\beta\alpha\sigma}{\Omega_1 \omega^{1-\sigma}Y_1^{1-\sigma}} \E \left[Y_2^{1-\sigma}\frac{(1-\alpha)e_2(1-s_{T-1}'(y_2)\Omega_2) + y_2^2 \gamma_{T-1}'(y_2)}{(y_2 - s_{T-1}(y_2)\Omega_2)^{1+\sigma}}\right],\notag
\end{gather}
and returning this to (\ref{eqn:GammaDerivativeUnsimplified}) we have
\begin{gather}
 \gamma_T' =  \frac{\frac{\beta\alpha\sigma}{\Omega_1 \omega^{1-\sigma}Y_1^{1-\sigma}}\E \left[Y_2^{1-\sigma}\frac{(1-\alpha)e_2(1-s_{T-1}'(y_2)\Omega_2) + y_2^2 \gamma_{T-1}'(y_2)}{(y_2 - s_{T-1}(y_2)\Omega_2)^{1+\sigma}}\right]}{\left[\frac{\sigma}{(1-\gamma_T)^{1+\sigma}} - \frac{\beta\alpha^2}{\Omega_1^2}\omega^{1+\sigma}\frac{1}{Y_1^{1-\sigma}}\E  V_{T-1}''\left(y_2;Y_2\right)\right]}\notag\label{eqn:GammaDerivativeEquality}
\end{gather}

Concavity of the value function implies that both terms in the denominator of the last expression are positive.  Estimating their sum by the former term, we get
\begin{gather}
\gamma_T' \leq \frac{2\beta\alpha}{\Omega_1 \omega^{1-\sigma}Y_1^{1-\sigma}}\E \left[Y_2^{1-\sigma}\frac{(1-\alpha)e_2(1-s_{T-1}'(y_2)\Omega_2) + y_2^2 \gamma_{T-1}'(y_2)}{y_2^{1+\sigma}(1 - \gamma_{T-1}(y_2))^{1+\sigma}}\right]\notag
\end{gather}
Concavity of the value function and the fact that $s_{T-1}$ is increasing implies that 
\begin{gather}
0< 1-s_{T-1}'(y_2)\Omega_2 <1\notag
\end{gather}
and in particular that $1-s_{T-1}'(y_2)\Omega_2=\mathcal{O}(1)$.  We therefore see from (\ref{eqn:GammaDerivativeEquality}) that $\gamma_T'$ itself is positive.  Moreover, absorbing factors not dependent on $N$ and $\omega$ into a single constant, we can rewrite the above bound as
\begin{gather}
\gamma_T' = \mathcal{O}\left(\frac{1}{\omega^{1-\sigma}}\E \frac{e_2 + y_2^2 \gamma_{T-1}'(y_2)}{y_2^{1+\sigma}}\right)\notag
\end{gather}
where we have used the fact that $\gamma_{T-1}$ is bounded away from 1 and the fact that aggregate shocks are bounded. 

\subsection{Proof of the Main Theorem}

Equipped with \ref{thm:GammaDerivative}, we may proceed to establish the approximate aggregation estimates \ref{thm:ApproximateAggregation}.  We wish to group our agents into bins such that any two agents within a single bin have approximately the same value of $\gamma_{T}(\omega;\{e_{j,t}\}_{t=2}^T,\{z_t\}_{t=2}^T,\{\Omega_t\}_{t=1}^{T-1},Y_{1})$.  For the current discussion, we further simplify our notation to read $\gamma_T(\omega) := \gamma_T(\omega,Y_1)$.  

Fix the number of time periods $T$ and the error parameter $\epsilon>0$.  The case in which any agent has all the wealth is trivial, so we may suppose not.  We begin by using the similar future prospects condition on the employment shocks to separate the agents into sets indexed by $A_1,...,A_s \subset \{1,...,N\}$ so that for $j,k\in A_l$ we have that the distributions of $e_{j,t}$ and $e_{k,t}$ are the same in each time period.                                                                                                                                                                                                                          

Next, for a fixed choice of the set $A_l$, we estimate the total variation of $\gamma_{T}$ as $\omega$ ranges across the interval $[0,1]$. 

\subsubsection{Variation of the Savings Function}
We proceed by induction. When $T=2$, $\gamma_{T-1} = \gamma_1 \equiv 0$.  Using this in our bound we get
\begin{gather}
\gamma_2' = \mathcal{O}\left(\frac{1}{\omega^{1-\sigma}}\E \frac{e_2}{y_2^{1+\sigma}}\right). \notag
\end{gather}
Integrating this bound over shares in $[0, 1]$, we can bound the total variation of $\gamma_2$ over this interval by
\begin{gather}
 \int_{0}^{1} \gamma_{2}'(\omega)\; d\omega = \mathcal{O}\left(\E \int_{0}^{1} \frac{e_2}{\omega^{1-\sigma}y_2^{1+\sigma}}\right)\; d\omega\notag\\
= \mathcal{O}\left(\E \left[\int_{0}^{e_2}\frac{e_2}{\omega^{1-\sigma}y_2^{1+\sigma}}\; d\omega + \int_{e_2}^1 \frac{e_2}{\omega^{1-\sigma}y_2^{1+\sigma}}\; d\omega \right]\right).\notag
\end{gather}

In the first term under the expected value, we bound $y_2 = \mathcal{U}(e_2)$, and so the contribution is
\begin{gather}
 \int_{0}^{e_2}\frac{e_2}{\omega^{1-\sigma}y_2^{1+\sigma}}\; d\omega = \mathcal{O}\left(\frac{1}{e_2^\sigma}\int_{0}^{e_2}\frac{1}{\omega^{1-\sigma}}\right)\notag\\
=\mathcal{O}\left(\frac{1}{e_2^\sigma}e_2^\sigma\right)\notag\\
=\mathcal{O}(1).\notag
\end{gather}

In the second term under the expected value, we instead bound $y_2 = \mathcal{U}(\omega)$, giving
\begin{gather}
 \int_{e_2}^1 \frac{e_2}{\omega^{1-\sigma}y_2^{1+\sigma}}\; d\omega  = \mathcal{O}\left(e_2 \int_{e_2}^1 \frac{1}{\omega^2}\; d\omega\right) \notag\\
 = \mathcal{O}\left(e_2\left(\frac{1}{e_2} - 1\right)\right)\notag\\
=\mathcal{O}(1).\notag
\end{gather}

Combining the last three calculations we get
\begin{gather}
 \int_{0}^{1} \gamma_{2}'(\omega)\; d\omega = \mathcal{O}(1)\notag
\end{gather}
as desired.  This concludes the $T=2$ case.

Let $T>2$, and suppose that the total variation bound has been established for $T-1$.  We now have
\begin{gather}
\int_{0}^{1} \gamma_{T}'(\omega)\; d\omega = \mathcal{O}\left(\E \int_{0}^{1} \left(\frac{e_2}{\omega^{1-\sigma}y_2^{1+\sigma}}+\frac{y_2^{1-\sigma} \gamma_{T-1}'(y_2)}{\omega^{1-\sigma}}\right)\; d\omega\right).\notag
\end{gather}
The first term under the integral is identical to the expression bounded in the $T=2$ case, and we need only use the inductive hypothesis to bound the second term.  Doing so requires that we have an estimate $\omega\gamma_{T-1}'(\omega) = \mathcal{O}(1)$.  This is accomplished by yet another straightforward induction, which we now pause to establish.

We have
\begin{gather}
 \omega\gamma_2'(\omega) = \mathcal{O}\left(\E\frac{e_2\omega^\sigma}{y_2^{1+\sigma}}\right)\notag\\
 = \mathcal{O}\left(\E\frac{e_2\omega^\sigma}{(\omega + e_2)^{1+\sigma}}\right)\notag\\
 = \mathcal{O}(1)\notag
\end{gather}
using the simple bound $e_2\omega^\sigma < (\omega + e_2)^{1+\sigma}$.  Proceeding inductively, assuming the bound for $\omega\gamma_{T-2}'(\omega)$, we have
\begin{gather}
 \omega\gamma_{T-1}'(\omega) = \mathcal{O}\left(\E\left[\frac{e_2\omega^\sigma}{y_2^{1+\sigma}}+\omega^\sigma y_2^{1-\sigma}\gamma_{T-2}'(y_2)\right]\right)\notag\\
= \mathcal{O}\left(\E\left[1+\frac{\omega^\sigma}{y_2^{\sigma}}\right]\right)\notag\\
= \mathcal{O}(1)\notag
\end{gather}
where in the second line we used the induction to bound $y_2^{1+\sigma}\gamma_{T-2}'(y_2)$.

Now, on the range $[0, e_2]$, we apply the above argument and bound $y_2$ as having size of order $e_2$, giving
\begin{gather}
  \int_{0}^{e_2} \frac{y_2^{1-\sigma} \gamma_{T-1}'(y_2)}{\omega^{1-\sigma}}\; d\omega = \mathcal{O}\left(e_2^{-\sigma}\int_{0}^{e_2} \frac{1}{\omega^{1-\sigma}}\; d\omega \right)\notag\\
 = \mathcal{O}\left(1\right)
\end{gather}

On the range $[e_2,1]$, we bound $y_2$ as having size of order $\omega$, giving
\begin{gather}
  \int_{e_2}^{1} \frac{y_2^{1-\sigma} \gamma_{T-1}'(y_2)}{\omega^{1-\sigma}}\; d\omega = \mathcal{O}\left(\int_{e_2}^{1} \gamma_{T-1}'(y_2)\; d\omega \right)\notag
\end{gather}
Since we have a $\mathcal{U}(1)$ bound on $\gamma_T$ and since $\gamma_T$ is increasing, we also have a $\mathcal{U}(1)$ bound on the derivative of $y_2$ with respect to $\omega$, which is given by $\frac{\alpha}{\Omega_1}(\gamma_T+\omega\gamma_T')$.  We can therefore change variables in the right hand side of the last line to get
\begin{gather}
  \int_{e_2}^{1} \frac{y_2^{1-\sigma} \gamma_{T-1}'(y_2)}{\omega^{1-\sigma}}\; d\omega = \mathcal{O}\left(\int_{e_2}^{1} \gamma_{T-1}'(y_2)\; dy_2 \right)\notag\\
=\mathcal{O}(1)\notag
\end{gather}
by the inductive hypothesis.

Combining the above calculations we see that
\begin{gather}
 \int_{0}^1 \frac{y_2^{1-\sigma} \gamma_{T-1}'(y_2)}{\omega^{1-\sigma}}\; d\omega = \mathcal{O}\left(1\right)\notag
\end{gather}
and consequently that 
\begin{gather}
 \int_{0}^{1} \gamma_{T}'(\omega)\; d\omega = \mathcal{O}(1)\notag
\end{gather}
as desired.

\subsection{Binning}
The previous calculations show that, for a fixed choice of $A_l$ (and hence the distributions of employment shocks), $\gamma_{T}'$ changes by at most $\mathcal{O}(1)$ as $\omega$ ranges across the interval $[0,1]$.  
We can therefore subdivide $[0,1]$ into subintervals $I_{m,l}=[a_{m,l},b_{m,l})$, $m=1,...,M$ such that
\begin{gather}
 M = O(1/\epsilon)\notag\\
 {\rm For\ each\ } l,\ {\rm for\ } \omega_{j_1},\omega_{j_2} \in I_{m,l}\ {\rm we\ have\ } |\gamma_{T}(\omega_{j_1})-\gamma_{T}(\omega_{j_2})|\leq \epsilon.\notag
\end{gather}
where $M$ does not depend on $N$ or the distribution of wealth.  Referencing now the wealth of the $j$th agent as $\omega_j$, we set
\begin{gather}
 B_{m,l} = \{j\in\{1,...,N\}\cap A_l:\ \omega_j\in I_{m,l} \}\notag\\
 \gamma_{m,l} = \gamma_{T} (a_{m,l})\notag\\
 Y_{m,l} = \sum_{j\in B_{m,l}} \omega_j Y_1 \notag.
\end{gather}

Given these definitions, we have 
\begin{eqnarray*}
 |\sum_{j\in B_{m,l}} \gamma_{T}(\omega_j)\omega_j Y_1 - \gamma_{m,l} Y_{m,l}| &=& |\sum_{j\in B_{m,l}} (\gamma_{T}(\omega_j)-\gamma_{m,l})\omega_j Y_1|\\ &\leq& \sum_{j\in B_{m,l}} |\gamma_{T}(\omega_j)-\gamma_{m,l}|\omega_j Y_1\\
& \leq & \epsilon Y_{m,l}\notag.
\end{eqnarray*}
Summing over $m=1,...,M$ and $l$, we get
\begin{gather}
|\sum_{j\in \{1,...,N\}} \gamma_{T}(\omega_j)\omega_j Y_1 - \sum_{m=1}^M\sum_{l} \gamma_{m,l} Y_{m,l} | \leq \epsilon \sum_{m=1}^M\sum_{l} Y_{m,l} \leq \epsilon Y_1.
\end{gather}
This is the estimate of the theorem.  Moreover, the number of groupings of agents is at most $sM = \mathcal{O}(1/\epsilon)$ as desired.

\bibliographystyle{econometrica}
\bibliography{C:/Research/Bib/OneBibToRuleThemAll}

\end{document}